\documentclass[a4paper,6pt]{article}
\usepackage[francais]{babel}
\usepackage{geometry}
\usepackage[T1]{fontenc}
\usepackage[latin1]{inputenc}
\usepackage{amsmath,amssymb,mathrsfs,amsthm}
\usepackage{soul}
\usepackage{epsfig}
\usepackage{hyperref}
\usepackage{graphicx}
\usepackage{xypic}
\usepackage{datetime}
\usepackage{fancyhdr}
\fancypagestyle{plain}{
\lhead{}
\rhead{}

}

\newtheorem{theorem}{Théorème}[section]
\newtheorem{proposition}[theorem]{Proposition}
\newtheorem{lemma}[theorem]{Lemme}
\newtheorem{conjecture}[theorem]{Conjecture}

\newtheorem{Corollaire}[theorem]{Corollaire}
\newtheorem{definition}[theorem]{Définition}

\newtheorem{remarque}[theorem]{Remarque}

\newcommand{\vc}{\|\cdot\|}
\newcommand{\C}{\mathbb{C}}

\newcommand{\Z}{\mathbb{Z}}
\newcommand{\Si}{\Sigma}
\newcommand{\s}{\mathbb{S}^1}
\newcommand{\lra}{\longrightarrow}

\newcommand{\la}{\lambda}

\newcommand{\R}{\mathbb{R}}
\newcommand{\cl}{\mathcal{C}^\infty}
\newcommand{\p}{\mathbb{P}}

\newcommand{\vf}{\varphi}
\newcommand{\si}{\sigma}

\newcommand{\N}{\mathbb{N}}
\newcommand{\z}{\overline{z}}
\newcommand{\pt}{\partial}

\newcommand{\h}{\mathcal{H}}

\title{Sur une inégalité fonctionnelle sur les variétés toriques avec  application à la torsion analytique holomorphe}
\date{}
\author{Mounir Hajli}
\begin{document}

\maketitle

\begin{abstract}
Soit $X$ une variété torique complexe projective non-singulière de dimension $n$, et $E$ un fibré en droites équivariant et ample
sur $X$. Notons par $\mathcal{H}_{E,0}$ l'espace des métriques hermitiennes $\cl$, positives et
invariantes par l'action du tore compact de $X$.

 Dans cet article, nous introduisons  $V_{\infty,\ast}$, une fonctionnelle sur $\mathcal{H}_{E,0}$,  nous
étudions ses propriétés, et  nous  la comparerons   à certaines fonctionnelles classiques.
Comme application,  nous étudions la variation de la torsion analytique holomorphe sur $\mathcal{H}_{E,0}$.

\end{abstract}
\tableofcontents

\section{Introduction}

Soit $X$ une variété torique complexe non-singulière, et $E$ un fibré en droites équivariant et ample sur $X$. Nous introduisons
une
nouvelle fonctionnelle sur l'espace de métriques admissibles et invariantes par l'action du tore compact de $X$ sur $E$. Grâce
à la structure combinatoire  de la variété, nous
établissons
que cette fonctionnelle est majorée. Nous la comparons à certaines fonctionnelles classiques, comme conséquence  nous retrouvons
quelques
résultats classiques.\\

Commençons par  rappeller  la construction de certaines fonctionnelles
classiques, ainsi que  les résultats associés.  Soit $X$ une variété compacte kählérienne,  et $E$ est un fibré en droites ample
sur $X$. Soit
 $\omega_0$ une forme kählérienne   dans la première
classe de Chern $c_1(E)$. Si l'on pose
\[
 \mathcal{H}_{\omega_0}=\bigl\{ u\in \mathcal{C}^\infty(X)\, \bigl|\, \omega_u:=dd^c u+\omega_0>0 \bigr\},
\]
alors cet ensemble s'identifie à l'ensemble des métriques  $\cl$ définies positives sur $E$. Classiquement, on
définit deux fonctionnelles, $\mathcal{E}_{\omega_0}$ et $\mathcal{L}_{\omega_0}$, sur
$\mathcal{H}_{\omega_0}$. La première fonctionnelle  $\mathcal{E}_{\omega_0}$, appelée la fonctionnelle
d'énergie  qui apparaît dans les travaux d'Aubin cf. \cite{Aubin2} et Mabuchi cf. \cite{Mabuchi}, elle
est définie comme suit:
\[
\mathcal{E}_{\omega_0}(u):=\frac{1}{(n+1)!\mathrm{Vol}(\omega_0)}\sum_{j=0}^n\int_X u\bigl(dd^c u+\omega_0
\bigr)^j\wedge
\omega_0^{n-j}\quad \forall\, u\in \h_{\omega_0}.
\]
On a, $\omega_0$ définit une métrique kählérienne sur $X$, et
donc sur $K_X$. Soit $u\in \mathcal{H}_{\omega_0}$, $u$ définit  une métrique sur $E$. Donc on dispose d'une métrique sur
$H^0(X,L\otimes K_X)$. La fonctionnelle $\mathcal{L}_{\omega_0}$ est alors définie comme suit:
\[
 \mathcal{L}_{\omega_0}(u):=-\frac{1}{N}\log \det\bigl(\bigl<s_i,s_j\bigr> \bigr)_{1\leq i,j\leq N }\quad \forall\, u\in \h_{\omega_0},
\]
avec $s_1,\ldots,s_N$ est un ensemble orthogonal de sections globales de $E\otimes K_X$  pour la métrique définie par
$\omega_0$ et  formant une base pour $H^0(X,E\otimes K_X)$, voir \cite{Berman}.

Si l'on pose,
\[
\mathcal{F}_{\omega_0}=\mathcal{E}_{\omega_0}-\mathcal{L}_{\omega_0}.
\]
Alors on montre, dans \cite{Berman}, le résultat suivant:
\begin{theorem}
Soit $K$ un groupe de Lie compact semi-simple  qui agit transitivement sur $X$, et $E$ un fibré en droites holomorphe $K$-homogène sur $X$.  Soit $\omega_0$ l'unique forme de Kähler invariante par l'action de $K$ sur $X$. On a,
\[
 \mathcal{F}_{\omega_0}(u)\leq 0\quad \forall \, u\in \mathcal{H}_{\omega_0}.\\
\]
\end{theorem}

Un des buts de cet article est présenter une nouvelle fonctionnelle notée $V_{\infty,E,m}$, définie dans le cadre des variétés
toriques non-singulières.  Nous montrerons que  le fait que $V_{\infty,E,m}$ soit bornée, implique une version faible du
théorème ci-dessus. Plus précisément, on établit que si $V_{\infty,E,m}$ est bornée alors $\mathcal{F}_{\omega_0}$ l'est aussi.

Notre second objectif  est motivé par la conjecture de Gillet-Soulé qui prédit le comportement du déterminant
régularisé (voir définition \eqref{rappeldet})
quand la métrique varie sur le fibré $E$, cf. \cite{Upper}. Lorsque $X=\p^1$,  Berman utilise \cite[corollaire 2]{Berman} pour
  montrer cette conjecture, c'est à dire  que le déterminant régularisé est majoré sur l'espace des métriques hermitiennes $\cl$ sur
  $\mathcal{O}(m)$, cf. \cite[§ 1.2.1]{Berman}. Notons qu'en dimension 1, cela  revient à dire que
  l'opposé de la
torsion analytique est majoré (voir \eqref{TmoinD}).  Nous étudions la variation de la torsion analytique holomorphe
moyennant la fonctionnelle  $V_{\infty,E,m}$, nous établissons un résultat qui décrit la variation de la torsion analytique
sur un espace de métriques sur un fibré en droites  sur une variété torique non-singulière quelconque, par exemple $\p^1$.
 Ce  résultat peut être vu  comme
une généralisation du \cite[corollaire 4]{Berman} dans le cadre des variétés toriques. \\

\noindent \textbf{Plan de l'article} :
Soit $X(\Si)$ une variété torique projective non singulière de dimension $n$, et $E$ un fibré en droites
ample
sur $X(\Si)$. Nous commençons par associer à toute variété torique non singulière,  une forme de volume
continue, notée $\omega_{E,\infty}^n$. Cette
construction est canonique, dans le sens qu'on n'a pas besoin de définir, à priori, une métrique kählerienne
sur le fibré tangent $TX(\Si)$ et que son expression est déterminée uniquement par la combinatoire de la variété, c'est
l'objet du théorème \eqref{formevolumecanonique}. L'idée clé de la preuve  utilise l'existence d'un recouvrement canonique
associé à toute
variété torique définie par un éventail. \\

Soit $h$ une métrique hermitienne continue  sur $E$. On définit
un produit hermitien sur  $A^{(0,0)}\bigl(X(\Si),E\bigr)$, l'espace des fonctions $\cl$ à coefficients dans $E$,
comme suit:
\[
 (s,t)_{L^2_{h,\infty_E}}=\int_{X(\Si)}h(s,t)\omega_{E,\infty}^n,
\]
pour tous $s,t\in A^{(0,0)}\bigl(X(\Si),E\bigr)$.\\

On note par $\h_{E,0}$, l'espace des métriques admissibles et invariantes par l'action du  tore compact de $X(\Si)$ sur $E$.

Soit $m\in \N^\ast$, on introduit  la fonctionnelle $V_{\infty,E,m}$ sur $\h_{E,0}$, définie comme suit (voir paragraphe \eqref{lafonctionelle}):
\begin{definition}
Soit $X(\Si)$ une variété torique projective complexe non singulière. Pour tout $m\in
\N_{\geq 1}$, on définit sur $\mathcal{H}_{E,0}$  la  fonctionnelle suivante:
\[
V_{\infty,E,m}(h):=\log h_{L^2,h^m,\infty_E}+2\int_{X(\Si)}\widetilde{\mathrm{ch}}\bigl(E^m,h^m,h_\infty^m
\bigr)\quad \forall \, h\in {\h}_{E,0}.
\]
où $h_\infty$ est la métrique canonique de $E$, $h_{L^2,h^m,\infty_E}$ le volume $L^2$ de $H^0(X(\Si),E)$ induit par
$\omega_{E,\infty}^n$ et $h^m$, et $\widetilde{\mathrm{ch}}\bigl(E^m,h^m,h_\infty^m
\bigr)$ est la classe de Bott-Chern associée au caractère $\mathrm{ch}$. Lorsque $X(\Si)=\p^n$, et $E=\mathcal{O}(1)$, on notera
 cette fonctionnelle par $V_{\infty,m}$.\\
\end{definition}

Nous montrons que la fonctionnelle $V_{\infty,E,m}$ est majorée sur ${\h}_{E,0}$, voir (théorème \eqref{borneVmtheorem}):
\begin{theorem}
Pour tout $m\in \N_{\geq 2}$, il existe une constante $c_m$  (resp. $c_{E,m}$) qui dépend uniquement de $m$ telle que
\[
 V_{\infty,m}(h)\leq c_m\quad\forall\, h\in \h_0\cap\{ h\leq h_\infty\bigr\},
\]
(resp.)
\[
 V_{\infty,E,m}(h)\leq c_{E,m}\quad\forall\, h\in {\h}_{E,0}\cap\{ h\leq h_\infty\bigr\}.
\]

Lorsque $m\gg 1$, alors
\[
 V_{\infty,m}(h)\leq c_m\quad\forall\, h\in \h_0,
\]
(resp.)
\[
 V_{\infty,E,m}(h)\leq c_{E,m}\quad\forall\, h\in {\h}_{E,0}.
\]

\end{theorem}

Comme application, nous  montrons  un résultat donnant la variation de la torsion analytique holomorphe lorsque
la métrique varie sur $\h_0$ (resp. $\h_{E,0}$):
\begin{theorem}
Soit $\p^n$ (resp. $X(\Si)$ une variété torique projective complexe non singulière de dimension $n$) muni
d'une métrique de kähler $\omega$.

Il existe $m_n\in \N$, tel que pour tout $h\in \h_0$ (resp. $h\in \h_{E,0}$) et pour
tout $m\geq m_n$:
\begin{align*}
 -T\bigl((\p^n,\omega);(\mathcal{O}(m),h^m) \bigr)
\leq-{c_m'},
\end{align*}
(resp.)
\begin{align*}
 -T\bigl((X(\Si),\omega);(E^m,h^m) \bigr)
\leq-{c_m'},
\end{align*}

où ${c_m'}$ est une constante réelle qui dépend uniquement de $m$ et de $\omega$.\\
\end{theorem}

Le paragraphe \eqref{CompBerman} est consacré à la comparaison entre $V_{\infty,\ast}$ (définition \eqref{newfunctional})
et  $\mathcal{F}_{\omega_0}$ la fonctionnelle  étudiée par Berman dans \cite{Berman}. En
particulier, on retrouve une version faible d'un résultat de \cite{Berman}. \\

La plupart des démonstrations sont données dans le cas de $\p^n$, mais elles s'étendent naturellement au
cas d'une variété torique projective complexe non singulière.

\section{Une inégalité fonctionnelle canonique sur une variété torique non-singulière }

Cette section est consacrée à l'introduction d'une nouvelle fonctionnelle, notée $V_{\infty,E,m}$, définie
 sur un espace de métriques sur
un  fibré en droites sur  variété torique non-singulière, voir la définition \eqref{newfunctional}, et à
son étude. Cette fonctionnelle a la
particularité d'être canoniquement associée à la variété.

Nous  commençons tout d'abord par établir qu'on peut
associer canoniquement à toute variété torique complexe non singulière
une forme de volume, que nous explicitons localement  en terme de la combinatoire de la variété, c'est le
but du théorème \eqref{formevolumecanonique}. \\

Il est bien connu, voir par exemple \cite{Demazure} ou \cite{Fulton}, qu'une  variété torique projective
normale sur $\C$, correspond
à la donnée d'un $\Z$-module libre $N$ et un éventail $\Si$ sur $N_\R=N\otimes_\Z \R$, et que tout
diviseur de Cartier équivariant
définit un polytope convexe dans $M_\R$, où $M$ est le dual de $N$. Il est possible de procéder autrement
en partant d'un polytope
convexe à sommets entiers et de lui associer une variété torique projective normale, voir
théorème \eqref{torassopoly} ci-dessous. Nous  adopterons cette
dernière construction, de
sorte que toute variété torique considérée dans cet article  soit attachée à un polytope
convexe donné.

\begin{theorem}\label{torassopoly}
Soit $\Delta\subset \R^n$ un polytope convexe d'intérieur non vide dont les sommet sont
dans
$\Z^n$. Il existe un unique éventail complet $\Si$ dans $\R^n$ et un unique diviseur de
Cartier
$E$ horizontal $T$-invariant sur $X(\Si)$ tels que:
\begin{enumerate}
 \item $\Delta_E=\Delta$.
\item Le diviseur $E$ est ample.
\end{enumerate}
L'éventail $\Si$ est le plus petit éventail complet tel que la fonction d'appui
$\psi_\Delta$\footnote{$\psi_\Delta$  définit le polytope $\Delta$.} est
linéaire par morceaux relativement à $\Si$. De plus, $X(\Si)$ est lisse si et seulement
si le
polytope $\Delta$ est absolument simple; dans ce cas, le diviseur $E$ est très ample.
\end{theorem}
\begin{proof}
 Voir \cite[théorème 2.4.1]{Maillot}.
\end{proof}

 Dans \cite{Batyrev}, Batyrev et Tschinkel introduisent un recouvrement canonique pour toute variété torique projective lisse, nous
utiliserons ce recouvrement pour construire  une forme de volume continue canoniquement attachée à la variété. Commençons par rappeler
la construction de Batyrev et Tschinkel.

\begin{definition}\label{Batyrev}[Batyrev et Tschinkel]
Soit $X(\Si)$, une variété torique projective sur $\C$, définie par un éventail
$\Si$. Pour tout $\si\in \Si$, on définit $C_\si\subset X(\Si)$ de la façon suivante:
\[
C_\si=\bigl\{x\in X(\Si): \; \forall \, m\in \mathcal{S}_\si:=\check{\si}\cap M,\;
\chi^m
\,\text{est
régulier en x et}\,\, |\chi^m(x)|\leq 1 \bigr\}.
\]
\end{definition}

\begin{proposition}
 Pour tout $\si\in \Si$, on a $C_\si\subset U_\si$ et $C_\si$ est compact. De plus, si
$\tau,\tau'\in
\Si$, alors:
\[
 C_\tau\cap C_{\tau'}=C_{\tau\cap\tau'}.
\]

\end{proposition}
\begin{proof}
 Voir \cite[proposition 3.2.2]{Maillot}.
\end{proof}

Soit $\bigl\{m_1,\ldots,m_q \bigr\}$ une famille génératrice du semi-groupe $\mathcal{S}_\si$. On dispose de l'immersion fermée:

\begin{equation}\label{immCsi}
 \vf_\si:U_\si\lra \C^q
\end{equation}
 \begin{equation*}
\quad \quad\quad \quad x\lra \bigl(\chi^{m_1}(x),\ldots,\chi^{m_q}(x) \bigr).
\end{equation*}

On remarque que $\vf(C_\si)=\vf(U_\si)\cap B(0,1)^q$.

\begin{proposition}
Si $\Si$ est complet, alors les compacts $C_\si$ forment un recouvrement de $X(\Si)$ lorsque $\si$
parcourt $\Si_{\max}$.
\end{proposition}
\begin{proof}
 \cite[proposition 3.2.3]{Maillot}.
\end{proof}
\begin{remarque}\label{integpartition}
\rm{On suppose que $X(\Si)$ est lisse, donc en particulier $\Si$ est complet, dans ce cas, si on se donne $\mu$ une forme de volume sur $X$, alors pour toute fonction absolument intégrable $f$ sur $X$, on a
\[
 \int_X f\mu=\sum_{\si\in \Si_{\max}}\int_{C_\si}f\mu,
\]
car $C_\si\cap C_{\si'}=C_{\si\cap \si'}$ est de dimension inférieure à $\dim_\C(X(\Si))-1$,
lorsque $\si\neq \si'\in \Si_{\max}$, et donc  de mesure négligeable.}
\end{remarque}

Soit $\Delta$ un polytope absolument simple, et on considère
$X(\Si)$, la variété torique associée et $E$ comme dans \eqref{torassopoly}, donc $X(\Si)$ est non-singulière et $E$ est très ample.  On a alors le résultat suivant:

\begin{theorem}\label{formevolumecanonique}
Il  existe une
unique forme volume continue sur $X(\Si)$, notée $\omega_{E,\infty}^n$ telle que
\[
(\omega^n_{E,\infty})_{|_{\mathrm{int}(C_\si)}}=\vf_{\si}^\ast\bigl(\mu_0 \bigr)\quad \forall \,
\si\in\Si_{\max}.
\]
où $\mu_0$ désigne la forme volume standard sur $\C^n$, et $\mathrm{int}(\ast)$ est l'intérieur de
$\ast$.
En particulier, si $\Delta=\Delta_n$, (c'est à dire le polytope standard de $\R^n$), donc $X(\Si)=\p^n$, alors  on a
\[
(\omega_\infty^n)_{|_{\{x_0\neq 0\}}}=\biggl(\frac{i}{2\pi}\biggr)^n\frac{\prod_{j=1}^ndz_j\wedge
d\z_j}{\Bigl(\max\big(1,|z_1|,\ldots,|z_n| \bigr)\Bigr)^{2(n+1)}},
\]
avec $z_j=\frac{x_j}{x_0}$ pour $j=1,\ldots,n$.\\

\end{theorem}
\begin{proof}
Localement sur un ouvert assez petit $U$, on a $\omega_{E,\infty}^n=f \prod_{j=1}^ndy_j\wedge d\overline{y}_j$ où $y_1,\ldots,y_n$ est un système local de coordonnées sur $U$ et $f$ est une fonction positive et continue sur $U$. Donc l'unicité résultera du   $(1.)$.\\

On fera la preuve pour $\Delta=\Delta_n$, c'est à dire $X(\Si)=\p^n$. Notons par
$x_0,\ldots,x_n$ les coordonnées homogènes standards de $\p^n$, et posons sur l'ouvert
$\bigl\{x_0\neq 0 \bigr\}$,
$z_j=\frac{x_j}{x_0}$ pour $j=1,\ldots,n$. \\

On considère la forme de volume $\omega_\infty^n$ sur $\p^n$, donnée sur $\{x_0\neq 0\}$ par
\[
(\omega_\infty^n)_{|_{\{x_0\neq 0\}}}=\biggl(\frac{i}{2\pi}\biggr)^n\frac{\prod_{j=1}^ndz_j\wedge
d\z_j}{\Bigl(\max\big(1,|z_1|,\ldots,|z_n| \bigr)\Bigr)^{2(n+1)}}.
\]
Notons que cette dernière condition détermine de façon unique $\omega_\infty^n$. $\p^n$ est défini par la donnée d'un $\Z$-module $N$, libre de rang $n$ et d'un éventail
$\Si_{\p^n}$ (voir par exemple, \cite[§ 1.4]{Fulton}. On note par $M$, le $\Z$-module dual à $N$, on les identifie à $\Z^n$. On
note   par
$e_1,\ldots,e_n$ la
base standard de $\Z^n$, et $e_0:=-\sum_{j=1}^ne_j$. On pose,
\[
 \si_{j}:=\sum_{\substack{k=0\\
k\neq j}}^n\R^+ e_k\quad \forall\, j=0,\ldots,n,
\]
alors $\Si_{\p^n,\max}=\bigl\{ \si_0,\si_1,\ldots,\si_n\bigr\}$. Vérifions que
\begin{align*}
 \mathcal{S}_j&:=\check{\si}_j\cap M=\sum_{\substack{k=1\\
k\neq j}}^n\Z^+(e_k^\ast-e_j^\ast)+\Z^+(-e_j^\ast),\quad \text{et}\quad
\mathcal{S}_0:=\check{\si}_0 \cap M=\sum_{k=1}^n\Z^+ e_k^\ast.
\end{align*}
où $e_j^\ast$ est le dual de $e_j$, pour $j=1,\ldots,n$.

 Soit donc $j\neq 0$, si $m=\sum_{k=1}^n m_k e_k^\ast\in \check{\si}_j$, alors
\begin{align*}
 <m,e_k>&=m_k\geq 0\quad \forall \, k\notin\{ 0,j\},\\
<m,e_0>&=-\sum_{k=1}^n m_k\geq 0.
\end{align*}
Donc,
\begin{align*}
 m&=\sum_{k=1}m_k e_k^\ast\\
&= \sum_{\substack{k=1\\ k\neq j}}^nm_k (e_k^\ast-e_j^\ast)+(-\sum_{k=1}^n m_k)(-e_j^\ast)\in \sum_{\substack{k=1\\
k\neq j}}^n\Z^+(e_k^\ast-e_j^\ast)+\Z^+(-e_j^\ast). \\
\end{align*}
Alors, $ \check{\si}_j\cap M\subseteq \sum_{\substack{k=1\\
k\neq j}}^n\Z^+(e_k^\ast-e_j^\ast)+\Z^+(-e_j^\ast)$.

 Réciproquement si $m\in \sum_{\substack{k=1\\
k\neq j}}^n\Z^+(e_k^\ast-e_j^\ast)+\Z^+(-e_j^\ast)$, alors $m$ peut s'écrire comme suit:
\begin{align*}
 m&= \sum_{\substack{k=1\\ k\neq j}}^n z_k (e_k^\ast-e_j^\ast)+z_j(-e_j^\ast)
\end{align*}
avec $z_k\in \Z^+$ pour $k\geq 1$. On a alors
\begin{align*}
 <m,e_k>&=z_k\in \Z^+\quad \forall \,k\geq 1,\, k\neq j,\\
<m,e_0>&=z_j\in \Z^+.
\end{align*}
Par suite, $ \sum_{\substack{k=1\\
k\neq j}}^n\Z^+(e_k^\ast-e_j^\ast)+\Z^+(-e_j^\ast)\subseteq \check{\si}_j\cap M$.\\

Maintenant si $j=0$, alors on vérifie que $\check{\si}_0=\sum_{k=1}^n\Z^+e_k^\ast$.\\

Comme dans  \eqref{immCsi}, on considère pour tout $j=1,\ldots,n$:
\begin{align*}
 \vf_j:U_{\si_j}&\lra \C^n\\
x&\lra \bigl(\chi^{e_1^\ast-e_j^\ast}(x),\ldots,\chi^{e_{j-1}^\ast-e_j^\ast}(x),\chi^{-e_j^\ast}(x),\chi^{e_{j+1}^\ast-e_j^\ast}(x),\ldots,\chi^{e_{n}^\ast-e_j^\ast}(x) \bigr).
\end{align*}
et
\begin{align*}
 \vf_0:U_{\si_0}&\lra \C^n\\
x&\lra \bigl(\chi^{e_1^\ast}(x),\ldots,\chi^{e_{n}^\ast}(x) \bigr).
\end{align*}
En passant aux coordonnées $z_1,\ldots,z_n$, on obtient:
\begin{align*}
 \vf_j:U_{\si_j}&\lra \C^n\\
z&\lra
\Bigl(\frac{z_1}{z_j},\ldots,\frac{z_{j-1}}{z_j},\frac{1}{z_j},\frac{z_{j+1}}{z_j},\ldots,\frac{z_{n}}{z_j}
\Bigr),
\end{align*}
et
\begin{align*}
 \vf_0:U_{\si_0}&\lra \C^n\\
z&\lra \bigl(z_1,\ldots,z_n \bigr).\\
\end{align*}

Rappelons que lorsque $\C^n$ est muni de sa métrique standard, alors la forme volume  associée, $\mu_0$ est donnée par:
\[
 \mu_0=\biggl(\frac{i}{2\pi}\biggr)^n \prod_{k=1}^ndy_k\wedge d\overline{y}_k
\]
où $\{y_1,\ldots,y_n \}$ est un système local de coordonnées holomorphes sur $\C^n$.\\

Fixons $j\neq 0$ et  posons  $y_k=\frac{z_k}{z_j}$ si $k\neq j$ et $y_j=\frac{1}{z_j}$. Montrons que:
\[
{\omega_\infty^n}_{|_{\mathrm{int}(\mathrm{C}_j)}}=\biggl(\frac{i}{2\pi}\biggr)^n\vf_j^\ast\bigl( \prod_{k=1}^ndy_k\wedge d\overline{y}_k\bigr).
\]
En effet, sur l'intérieur de $C_j$, on a
{\allowdisplaybreaks
\begin{align*}
(\vf_j^{-1})^\ast \omega_\infty^n&= \biggl(\frac{i}{2\pi}\biggr)^n (\vf_j^{-1})^\ast \biggl(\frac{\prod_{k=1}^n
dz_k\wedge d\z_k}{\bigl(\max(1,|z_1|,\ldots,|z_n| ) \bigr)^{2(n+1)}}\biggr)\\
&=\biggl(\frac{i}{2\pi}\biggr)^n\frac{d(\frac{1}{y_j})\wedge d(\frac{1}{\overline{y}_j})\prod_{\substack{k=1\\ k\neq j}}^n d(\frac{y_k}{y_j})\wedge d(\frac{\overline{y_k}}{\overline{y}_j})}{\max\Bigl(1,\bigl|\frac{y_1}{y_j}\bigr|, \ldots,\bigl|\frac{y_{j-1}}{y_j}\bigr|,\bigl|\frac{1}{y_j}\bigr|, \bigl|\frac{y_{j+1}}{y_j}\bigr|, \ldots,\bigl|\frac{y_n}{y_j}\bigr| \Bigr)^{2(n+1)}}\\
&=\biggl(\frac{i}{2\pi}\biggr)^n\frac{|y_j|^{-2(n+1)}\prod_{k=1}^n d y_k\wedge d\overline{y_k}}{\max\Bigl(1,\bigl|\frac{y_1}{y_j}\bigr|, \ldots,\bigl|\frac{y_{j-1}}{y_j}\bigr|,\bigl|\frac{1}{y_j}\bigr|, \bigl|\frac{y_{j+1}}{y_j}\bigr|, \ldots,\bigl|\frac{y_n}{y_j}\bigr| \Bigr)^{2(n+1)}}\\
&=\biggl(\frac{i}{2\pi}\biggr)^n\frac{\prod_{k=1}^n d y_k\wedge d\overline{y_k}}{\max(1,\bigl|y_1\bigr|, \ldots, \ldots,\bigl|y_n\bigr| )^{2(n+1)}}\\
&=\biggl(\frac{i}{2\pi}\biggr)^n\prod_{k=1}^n d y_k\wedge d\overline{y_k}.
\end{align*}
}
Sur $C_0$,
 on a pose $y_k=z_k$ pour $k=1,\ldots,n$, alors
\[
{\omega_\infty^n}_{|_{\mathrm{int}(\mathrm{C}_0)}}=\biggl(\frac{i}{2\pi}\biggr)^n\vf_0^\ast\bigl( \prod_{k=1}^n dy_k\wedge d\overline{y}_k\bigr).
\]
\end{proof}

\subsection{La fonctionnelle $V_{\infty,\ast}$}\label{lafonctionelle}
On considère $\Delta_n$ le polytope standard dans $\R^n$ (resp. $\Delta$ un polytope
absolument
simple). Soit $\p^n$ (resp. $X(\Si)$) l'espace projectif complexe de dimension $n$
associé à
$\Delta_n$ (resp. la variété torique projective complexe non singulière associée à
$\Delta$) comme dans \eqref{torassopoly}. On considère la   forme volume singulière  $\omega_{\infty}^n$ (resp.
$\omega_{E,\infty}^n$)  sur $\p^n$ (resp. sur $X(\Si)$).\\

Soit $m\in \N_{\geq 1}$. On dispose  pour chaque métrique hermitienne continue $h$   sur
$\mathcal{O}(1)$ (resp. sur $E$),   de deux normes sur
$A^{(0,0)}\bigl(\p^n,\mathcal{O}(m)\bigr)$
(resp. sur  $A^{(0,0)}\bigl(X(\Si),E^{\otimes m}\bigr)$):
\begin{enumerate}
\item
La première  norme notée
par
$\vc_{L^2_{h^m,\infty}}$ (resp. $\vc_{L^2_{h^m,\infty_E}}$)  est  associée au produit hermitien
suivant:
\[
\bigl<s,t\bigr>_{L^2_{h^m,\infty}}=\int_{\p^n}(h^ m)\bigl(s,t\bigr)\omega_\infty^n\quad
\forall\,s,t\in
A^{(0,0)}\bigl(\p^n,\mathcal{O}(m)\bigr),
\]
(resp.)
\[
\bigl<s,t\bigr>_{L^2_{h^m,\infty_E}}=\int_{X(\Si)}(h^m)\bigl(s,t\bigr)\omega_{E,\infty}^n\quad
\forall\,s,t\in
A^{(0,0)}\bigl(X(\Si),E^m\bigr),
\]
 et on notera par $h_{L^2,h^m,\infty}$ (resp.$h_{L^2,h^m,\infty_E}$ ) la norme hermitienne induite sur l'espace $\det
 H^0(\p^n,\mathcal{O}(m))$ (resp. $\det H^0(X(\Si),E^m)$).

\item La deuxième norme notée  $\vc_{\sup}$ (resp. $\vc_{E,\sup}$),
est définie comme suit:
\[
 \|s\|_{\sup}=\sup_{x\in \p^n}\|s(x)\|^{\otimes m}\quad \forall \, s\in
A^{(0,0)}(\p^n,\mathcal{O}(m)).
\]
(resp.)
\[
 \|s\|_{E,\sup}=\sup_{x\in X(\Si)}\|s(x)\|^{\otimes m}\quad \forall \, s\in
A^{(0,0)}(X(\Si),E^m).
\]

\end{enumerate}

Rappelons qu'une métrique $h$ sur $E$ est dite admissible, si $h$ est une limite uniforme d'une suite de métriques positives et de classe
$\cl$.

On note par $\h_0$ (resp. ${\h}_{E,0}$), l'espace des métriques sur $\mathcal{O}(1)$ (resp. $E$ ) de classe  $\cl$, admissibles et invariantes par l'action du tore compact de $\p^n$ (resp. de $X(\Si)$).  On introduit la définition suivante:
\begin{definition}\label{newfunctional}
Soit $X(\Si)$ une variété torique projective complexe non singulière. Pour tout $m\in
\N_{\geq 1}$, on définit une fonctionnelle sur ${\h}_{E,0}$ en posant pour tout $h\in \h_{E,0}$:
\[
V_{\infty,E,m}(h):=\log h_{L^2,h^m,\infty_E}+2\int_{X(\Si)}\widetilde{\mathrm{ch}}\bigl(E^m,h^m,h_\infty^m \bigr),
\]
où $\widetilde{\mathrm{ch}}\bigl(E^m,h^m,h_\infty^m \bigr)$ est la classe de Bott-Chern généralisée \footnote{C'est à dire une
forme différentielle généralisée, voir \cite[§4.3]{Maillot}} associée à la suite exacte $0\rightarrow (E^m,h^m)\rightarrow (E^m,
h_\infty^m)\rightarrow 0$ et au caractère $\mathrm{ch}$.
\end{definition}

\begin{remarque}\label{remarqueinvx}
\leavevmode
\begin{enumerate}
\rm{
\item Contrairement à la fonctionnelle $\mathcal{F}_{\omega_0}$ introduite dans \cite{Berman}, voir \eqref{bermanfonctionnelle},
la fonctionnelle $V_{\infty,E,m}$ n'est pas invariante par multiplication par  des constantes strictement positives.
Plus précisement, si $t>0$  on a
\begin{align*}
V_{\infty,E,m}(t\,h)&=\biggl(\binom{n+m}{n}-2\frac{m^{n+1}}{n!} \biggr)\log(t)+V_{\infty,E,m}(h)\quad \forall \, h\in
\h_0.
\end{align*}
\item $V_{\infty,E,m}$ ne dépend pas à priori d'un choix de métrique kählérienne sur $X(\Si)$; ce qui n'est pas le cas
de $\mathcal{F}_{\omega_0}$.}

\end{enumerate}

\end{remarque}

Dans la suite, on suppose que  $\Delta=\Delta_n$ c'est à dire on se restreindra  au cas de $\p^n$.\\

On se propose d'étudier la fonctionnelle $V_{\infty,m}$, Pour cela, on va traduire cette étude  en un problème de  la géométrie convexe sur $\R^n$, par l'intermédiaire de la correspondance suivante: Soit $h\in \h_0$, on  lui associe la fonction suivante:
\[
 f_h(u)=\log\vc \bigl(\exp(-u) \bigr)\quad \forall \, u\in \R^n.
\]
où $\exp(-\cdot)$
\begin{align*}
\R^n&\lra \p^n\\
u=(u_1,\ldots,u_n)&\lra \exp(-u):=\bigl[1:\exp(-u_1),\ldots,\exp(-u_n)\bigr].
\end{align*}
Rappelons que cette fonction a été introduite par Burgos, Philippon et Sombra dans \cite{Burgos2} afin
d'étudier les
hauteurs des variétés toriques projectives. On montre, voir \cite{Burgos2}, que $h\mapsto f_h$ définit
une correspondance bijective entre
les éléments de $\h_0$ et une classe de  fonctions concaves de classe $\cl$ sur
$\R^n$.

 On appelle la
transformée de Legendre-Fenchel de $f_h$, la fonction suivante:
\begin{align*}
 \check{f}_h(x)&=\inf_{u\in \R^n}\bigl(<x,u>-f_h(u) \bigr)\quad \text{si}\;x\in
\Delta_n,\\
\check{f}_h(x)&=-\infty\quad \text{sinon}.
\end{align*}

Comme $h$ est invariante par $(\s)^n$ alors
\[
h_{L^2,h^m,\infty}=\prod_{ \nu\in \mathcal{P}_m}
\bigl<z^\nu,z^\nu\bigr>_{L^2_{h^m,\infty}},
\]
où $ \mathcal{P}_m=\bigl\{\nu=(\nu_1,\ldots,\nu_n)\in \N^n\;\bigl|\;
\sum_{j=1}^n\nu_j\leq m\bigr\}$ et $\bigl\{z^\nu \,\bigl|\; \nu\in \mathcal{P}_m\bigr\}$
est une base de $H^0(\p^n,\mathcal{O}(m))$ formée par les monômes de degré inférieur à
$m$.\\

Le résultat suivant sera utile dans la suite:
\begin{theorem}\label{bottchernfenchel}
Soit $\p^n$ l'espace projectif de dimension $n$ (resp. $X(\Si)$  une variété torique projetive non-singulière) sur $\C$.
$\overline{\mathcal{O}(1)}=(\mathcal{O}(1),h)$ (resp. $\overline{E}=(E,h)$) le fibré de Serre (resp. un fibré en droites
équivariant) muni
d'une métrique $h$, positive et invariante par l'action du tore compact de $\p^n$ (resp. $X(\Si)$). On a
\[
\int_{\p^n}\widetilde{\mathrm{ch}}(\mathcal{O}(1),h,h_\infty)=\int_{\Delta_n} \check{f}_h
\]
(resp.)
\[
\int_X\widetilde{\mathrm{ch}}(E,h,h_\infty)=\int_{\Delta}\check{f}_h.
\]
où $h_\infty$ est la métrique canonique sur $\mathcal{O}(1)$ (resp. sur $E$).
\end{theorem}
\begin{proof}
See \cite[§ 9.3]{Nystrom}.
\end{proof}

Pour tout $\nu=(\nu_1,\ldots,\nu_n )\in \N^n$ tel que $\sum_{k=1}^n\nu_k\leq m$, on
considère l'ensemble:
\[
\Delta_{n,\nu}=\bigl\{x\in \Delta_n\,\bigl|\; \frac{\nu_j}{m}\leq
x_j<\frac{\nu_j+1}{m}\quad \forall\,
j=1,\ldots,n \bigr\}.
\]
On vérifie que la famille formée par ces ensembles  forme une partition de $\Delta_n$. On
considère les polynômes $R_\nu$ suivants:

\[R_\nu(x)=\bigl(1-\sum_{k=1}^n\nu_k+m\sum_{k=1}^n x_k\bigr)\prod_{k=1}^n
\bigl(1+\nu_k-mx_k\bigr).\]

On notera par $\psi$ la fonction définie sur l'ouvert $\bigl\{x_0\neq 0 \bigr\}$ par:
\[
 \psi(z)=\frac{1}{2}\log h(1,1).
\]
Dans ce cas $h(z^\nu,z^\nu)=|z|^{2\nu}e^{2m\psi(z)}$ sur $\bigl\{x_0\neq 0 \bigr\}$ et que
$f_h(u)=\psi(\exp(u))$ pour tout $u\in \R^n$.\\

\begin{proposition}\label{inegaliteznu}
Soit $h\in \h_0$. On a
\begin{equation}
\log \bigl<z^\nu,z^\nu\bigr>_{L^2_{h^m,\infty}}\leq -2m \frac{1}{\mathrm{Vol}(\Delta_{n,\nu})}\int_{\Delta_{n,\nu}}\check{f}_h(x)dx-\frac{1}{\mathrm{Vol}(\Delta_{n,\nu})}\int_{\Delta_{n,\nu}} \log R_\nu(x)dx+\log(n+1),
\end{equation}
pour tout $h\in \h_0$, $m\in \N_{\geq 1}$ et $\nu\in \mathcal{P}_m$. Le terme
à droite est fini.
\end{proposition}
\begin{proof}
Fixons $\nu$ comme avant et soit $x\in \Delta_{n,\nu}$, on a
{\allowdisplaybreaks
\begin{align*}
 \bigl<z^\nu,z^\nu\bigr>_{L^2_{h^m,\infty}}&=\sum_{\si \in
\Si_{\max}}\int_{C_\si}h^{\otimes
m}(z^\nu,z^\nu)\omega_\infty^n\quad\text{d'après}\; \eqref{integpartition}\\
&=\sum_{j=0}^n\int_{C_j}h^{\otimes m}(z^\nu,z^\nu)\omega_\infty^n\\
&=\sum_{j=0}^n\int_{C_j}|z|^{2\nu}e^{2m \psi(z)}\omega_\infty^n\\
&=\sum_{j=0}^n\int_{C_j}|z|^{2\nu-2mx}|z|^{2mx}e^{2m \psi(z)}\omega_\infty^n\\
&\leq \sum_{j=0}^n\sup_{z\in C_j}\Bigl(|z|^{2mx}e^{2m\psi(z)}\Bigr)\int_{C_j}|z|^{2\nu-2mx}\omega_\infty^n\\
&\leq \sum_{j=0}^n\sup_{z\in \p^n}\Bigl(|z|^{2mx}e^{2m\psi(z)}\Bigr)\int_{C_j}|z|^{2\nu-2mx}\omega_\infty^n\\
&=\sum_{j=0}^n\sup_{u\in \R^n}\Bigl(e^{(-2<mx,u>+ 2m f(u))}\Bigr)\int_{C_j}|z|^{2\nu-2mx}\omega_\infty^n\\
&=\sum_{j=0}^n e^{-2m\inf_{u\in \R^n}(<x,u>- f(u))}\int_{C_j}|z|^{2\nu-2mx}\omega_\infty^n\\
&=\sum_{j=0}^n e^{-2m \check{f}_h(x)}\int_{C_j}|z|^{2\nu-2mx}\omega_\infty^n\\
&= \bigl(\frac{i}{2\pi} \bigr)^n e^{-2m\check{f}_h(x)}\int_{C_0}|z|^{2\nu-2mx} \prod_{j=1}^ndz_j\wedge
d\z_j+\sum_{j=1}^n e^{-2m \check{f}_h(x)}\int_{C_j}|z|^{2\nu-2m x}\omega_\infty^n.
\end{align*}}
Calculons les intégrales qui figurent dans la dernière ligne: Si $j\neq 0$, en utilisant
les changements de variables évidents, on a:
\begin{align*}
 \int_{C_j}|z|^{2\nu-2mx}
\omega_\infty^n&=\bigl(\frac{i}{2\pi}\bigr)^n\int_{C_j}\prod_{k=1}^n|z_k|^{2\nu_k-2m
x_k}\omega_\infty^n\\
&=\bigl(\frac{i}{2\pi}\bigr)^n\int_{D^n}\prod_{\substack{k=1\\ k\neq
j}}^n\bigl|\frac{y_k}{y_j}\bigr|^{2\nu_k-2m
x_k}\bigl|\frac{1}{y_j}\bigr|^{2\nu_j-2mx_j} \prod_{k=1}^n dy_k\wedge d\overline{y}_k\\
&=\int_{[0,1]^n}\prod_{\substack{{k=1}\\ k\neq j }}^nr_k^{\nu_k-mx_k}
r_j^{(-\sum_{k=1}^n\nu_k+m\sum_{k=1}^n x_k)}\prod_{k=1}^n dr_k\\
&=\frac{1}{1-\sum_{k=1}^n\nu_k+m\sum_{k=1}^n x_k}\prod_{\substack{k=1\\ k\neq j}}^n
\frac{1}{(1+\nu_k-mx_k)}.
\end{align*}
Pour $j=0$, on obtient:
\begin{align*}
 \int_{C_0}|z|^{2\nu-2m x}\omega_\infty^n=\prod_{k=1}^n \frac{1}{(1+\nu_k-mx_k)}
 \end{align*}
Notons que comme $x\in \Delta_{n,\nu}$, alors toutes ces intègrales sont convergentes.\\

Donc,
\begin{align*}
\bigl<z^\nu,z^\nu\bigr>_{L^2_{h^m,\infty}}&\leq e^{-2m\check{f}_h(x)}\Bigl(\prod_{k=1}^n \frac{1}{(1+\nu_k-mx_k)}+
\sum_{j=1}^n \frac{1}{1-\sum_{k=1}^n\nu_k+m\sum_{k=1}^n x_k}\prod_{\substack{k=1\\ k\neq j}}^n
\frac{1}{(1+\nu_k-mx_k)}
\Bigr)\\
&=e^{-2m\check{f}_h(x)}\frac{n+1}{1-\sum_{k=1}^n\nu_k+m\sum_{k=1}^n x_k}\prod_{k=1}^n
\frac{1}{(1+\nu_k-mx_k)}.
 \end{align*}
Récapitulons, en posant $R_\nu(x)=\bigl(1-\sum_{k=1}^n\nu_k+m\sum_{k=1}^n x_k\bigr)\prod_{k=1}^n
\bigl(1+\nu_k-mx_k\bigr)$, on a montré que
\begin{equation}\label{inegalitehfr}
\log \bigl<z^\nu,z^\nu\bigr>_{L^2_{h^m,\infty}}\leq -2m\check{f}_h(x)+\log(n+1)-\log R_\nu(x),
\end{equation}
pour tout $\nu=(\nu_1,\ldots,\nu_n)\in\N^n $ avec $ \sum_{j=1}^n\nu_j\leq m$ et $ x\in \Delta_{n,\nu}$.

Par  définition de $R_\nu$, on vérifie que $\log R_\nu$ est absolument intégrable sur $\prod_{j=1}^n
\bigl[\frac{\nu_j}{m}, \frac{\nu_j+1}{m}\bigr]$, et donc sur $\Delta_{n,\nu}\subset \prod_{j=1}^n
\bigl[\frac{\nu_j}{m}, \frac{\nu_j+1}{m}\bigr]$.  \\

De \eqref{inegalitehfr} et en intégrant sur $\Delta_{n,\nu}$, on obtient:
\begin{align*}
\log \bigl<z^\nu,z^\nu\bigr>_{L^2_{h^m,\infty}}\leq -2m
\frac{1}{\mathrm{Vol}(\Delta_{n,\nu})}\int_{\Delta_{n,\nu}}\check{f}_h(x)dx-\frac{1}{\mathrm{Vol}(\Delta_{n
,\nu})}\int_{\Delta_{n,\nu}} \log R_\nu(x)dx+\log(n+1).
\end{align*}

\end{proof}

\begin{remarque}
\rm{Même si  $V_{\infty,m}$ n'est pas invariante (voir \eqref{remarqueinvx}), on note que si l'on multiplie $h$ par $t>0$,  alors on a,
\begin{equation}\label{invarianceFh}
-2m \check{f}_h(x)-\log \bigl<z^\nu,z^\nu\bigr>_{L^2_{h^m,\infty}}=-2m \check{f}_{th}(x)-\log
\bigl<z^\nu,z^\nu\bigr>_{L^2_{th,\infty}}\quad \forall\, t>0.\end{equation} }
\end{remarque}

\begin{theorem}\label{borneVmtheorem}
Pour tout $m\in \N_{\geq 2}$, il existe une constante $c_m$ qui dépend uniquement de $m$ telle que
\begin{equation}\label{vmh0}
 V_{\infty,m}(h)\leq c_m\quad\forall\, h\in \h_0\cap\{ h\leq h_\infty\bigr\}.
\end{equation}
Lorsque $m\gg 1$, alors
\begin{equation}\label{vmh}
 V_{\infty,m}(h)\leq c_m\quad\forall\, h\in \h_0.
\end{equation}

\end{theorem}
\begin{proof}
Commençons par montrer que la deuxième assertion découle de la première. En effet, on a pour tout $h\in \h_0$ et $t>0$ fixé:
\begin{align*}
V_{\infty,m}(h)&=V_{\infty,m}(t^{-1}th)\\
&=-\binom{n+m}{n}\log t+2 \frac{m^{n+1}}{n!}\log t +V_{\infty,m}(th)\\
&=\biggl(\binom{n+m}{n}-2\frac{m^{n+1}}{n!} \biggr)(-\log t)+V_{\infty,m}(th).
\end{align*}
Comme $\binom{n+m}{n}$ est un polynôme de degré $n$ en $m$, alors on peut trouver $m\gg 1$ qui ne dépend pas de $h$ tel que:
\begin{align*}
V_{\infty,m}(h)&\leq V_{\infty,m}(th)\quad \forall\, 0<t<1.
\end{align*}
Par compacité de $\p^n$, on peut trouver  $0<t_0<1$ tel que $t_0 h\leq h_\infty$. Donc \eqref{vmh} découle de \eqref{vmh0}.\\

Soit $h\in \h_0$ tel que $h\leq h_\infty$ et montrons \eqref{vmh0}. On a,
\[
  \check{f}_h(x)\geq  \check{f}_{h_\infty}(x)\quad\forall\, x\in \Delta_n.
\]
Or, on peut vérifier que  $ \check{f}_{h_\infty}(x)=0$ pour tout $x\in \Delta_n$. Donc,
\[
  \check{f}_h(x)\geq 0\quad \forall\,x\in \Delta_n\quad\forall \;h\in \h_0\quad\forall\,h\leq h_\infty.
\]
Notons que $\Delta_{n,\nu}\subset \prod_{j=1}^n \bigl[\frac{\nu_j}{m}, \frac{\nu_j+1}{m}\bigr]$, alors
$\mathrm{Vol}(\Delta_{n,\nu})\leq \frac{1}{m^n}${\footnote{ On a choisit une mesure de Lebesgue sur $\R^n$ , vérifiant $\mathrm{Vol}([0,1]^n)=1$}}. On a dans ce cas:
{\allowdisplaybreaks
\begin{align*}
\sum_{\nu\in  \mathcal{P}_m}\log \bigl<z^\nu,z^\nu&\bigr>_{L^2_{h^m,\infty}}\leq -2m\sum_{\nu\in
\mathcal{P}_m}\frac{1}{\mathrm{Vol}(\Delta_{n,\nu})}\int_{\Delta_{n,\nu}}\check{f}_h(x)dx-\sum_{\nu\in
\mathcal{P}_m}\frac{1}{\mathrm{Vol}(\Delta_{n,\nu})}\int_{\Delta_{n,\nu}} \log R_\nu(x)dx\\
&+\binom{n+m}{n}\log(n+1)\quad \text{d'après }\;\eqref{inegaliteznu}\\
&= -2m\sum_{\nu\in
\mathcal{P}_m}\frac{\mathrm{Vol}(\Delta_n)}{\mathrm{Vol}(\Delta_{n,\nu})}\frac{1}{\mathrm{Vol}(\Delta_n)}
\int_{\Delta_{n,\nu}}\check{f}_h(x)dx-\sum_{\nu\in
\mathcal{P}_m}\frac{1}{\mathrm{Vol}(\Delta_{n,\nu})}\int_{\Delta_{n,\nu}} \log R_\nu(x)dx\\
&+\binom{n+m}{n}\log(n+1)\\
&\leq -2m\sum_{\nu\in \mathcal{P}_m}\frac{m^n}{
n!}\frac{1}{\mathrm{Vol}(\Delta_n)}\int_{\Delta_{n,\nu}}\check{f}_h(x)dx-\sum_{\nu\in
\mathcal{P}_m}\frac{1}{\mathrm{Vol}(\Delta_{n,\nu})}\int_{\Delta_{n,\nu}} \log R_\nu(x)dx\\
&+\binom{n+m}{n}\log(n+1)\\
&= -2m^{n+1}\int_{\Delta_{n}}\check{f}_h(x)dx-\sum_{\nu\in
\mathcal{P}_m}\frac{1}{\mathrm{Vol}(\Delta_{n,\nu})}\int_{\Delta_{n,\nu}} \log R_\nu(x)dx\\
&+\binom{n+m}{n}\log(n+1).
\end{align*}}
Récapitulons, nous venons d'établir l'inégalité suivante:
\[
\log h_{L^2,h^m,\infty}+ 2 m^{n+1}\int_{\Delta_{n}}\check{f}_h(x)dx\leq -\sum_{\nu\in
\mathcal{P}_m}\frac{1}{\mathrm{Vol}(\Delta_{n,\nu})}\int_{\Delta_{n,\nu}} \log R_\nu(x)dx+\binom{n+m}{n}\log(n+1).
\]
Mais d'après  \eqref{bottchernfenchel}, on a
\[
\int_{\Delta_n}\check{f}_h(x)dx=\int_{\p^n}\widetilde{\mathrm{ch}}(\mathcal{O}(1),h,h_\infty).
\]

Si l'on pose
\[
c_m:=-\sum_{\nu\in \mathcal{P}_m}\frac{1}{\mathrm{Vol}(\Delta_{n,\nu})}\int_{\Delta_{n,\nu}} \log R_\nu(x)dx+\binom{n+m}{n}\log(n+1),
\]
alors on conclut que,
\[
V_{\infty,m}(h)=\log h_{L^2,h^m,\infty}+2\int_{\p^n}\widetilde{\mathrm{ch}}(\mathcal{O}(m),h^m,h_\infty^m)\leq c_m,\quad \forall \,m\in \N\quad
\forall\,
h\in \h_0\cap\{h\leq h_\infty\}.
\]
\end{proof}

Soit $X$ une variété torique non-singulière. On fixe $\omega_0=c_1(E,h_0)$ une forme de Chern positive dans la classe $c_1(E)$, avec $h_0$ invariante par l'action du tore compact.
Soit $h\in \h_0$. Soit $u:=-\log \frac{h}{h_0}$.
On pose:
\[
P_{\omega_0}[u](x)=\sup\bigl\{v(x)\, \bigl| v\in \h_{\omega_0}, \, v\leq u   \bigr\}\quad \forall \, x\in X.
\]
On note par $P[h]$ la métrique donnée par $h_0 \,e^{-P_{\omega_0}[u]}$.
\begin{lemma}
Si $h\in \h_0$, alors  $P[h]\geq h$ et $P[h]\in \h_{E,0}$.
\end{lemma}
\begin{proof}
Pour la positivité de $P[h]$ on peut consulter \cite[§1.4]{Berman}. 
Soit $\theta\in (\mathbb{S}^1)^n$ le tore compact de $X(\Si)$. On note par $\theta\cdot x$, l'action de $\theta$ sur un point $x$ de
$X(\Si)$. Par définition,  on a $P[u](\theta \cdot x)\geq v(\theta\cdot x)$, si $v\leq u$ et $v\in \h_{\omega_0}$.  Supposons que
$v\in \h_{\omega_0}$, $v\leq u$  et  posons
$v_\theta$ la fonction définie par $v_\theta(x)=v(\theta\cdot x)$ pour tout $x\in X(\Si)$. Il est clair que $v_\theta\in
\h_{\omega_0}$.

 On a $v_\theta(x)=v(\theta\cdot x)\leq u(\theta\cdot x)=u(x)$ ($u$ est invariante par l'action du tore compact).
Donc, $P[u](x)\geq v_\theta(x)$. Par suite, $P[u](x)\geq P[u](\theta\cdot x)$ $\forall\, \theta\in (\mathbb{S}^1)^n$. On déduit que
\[
P[u](\theta\cdot x)=P[u](x)\quad\forall\,\theta\in (\mathbb{S}^1)^n.
\]
Par suite, $P[h]\in \h_{E,0}$.
\end{proof}

Lorsque $X=\p^1$, alors on peut établir comme dans \cite[§ 3.3]{Berman}:
\[
V_{\infty,m}(h)\leq V_{\infty,m}(P[h]),
\]
pour toute métrique $h$ de classe $\cl$ et invariante par l'action de $\mathbb{S}^1$. En effet, en suivant \cite[3.3]{Berman}, on
obtient:
\[
\int_{\p^1}\widetilde{\mathrm{ch}}(\mathcal{O}(1),h,h_\infty)\leq \int_{\p^1}\widetilde{\mathrm{ch}}(\mathcal{O}(1),P[h],h_\infty),
\]

donc,
\begin{align*}
V_{\infty,m}(h)&=\log h_{L^2,h^m,\infty}+2\int_{\p^n}\widetilde{\mathrm{ch}}\bigl(\mathcal{O}(m),h^m,h_\infty^m \bigr)\\
&\leq \log h_{L^2,(P[h])^m,\infty}+2\int_{\p^n}\widetilde{\mathrm{ch}}(\mathcal{O}(m),(P[h])^m,h_\infty^m)\\
&=V_{\infty,m}(P[h]).
\end{align*}

\section{Applications à la torsion analytique holomorphe}
Dans ce paragraphe, on s'intéresse au comportement de la torsion analytique holomorphe
lorsque la métrique varie  dans $\h_0$. Cette étude est motivée par  une
conjecture de Gillet et Soulé qui prédit que le déterminant régularisé vu comme fonction
en la métrique est borné supérieurment, voir \cite{Upper}.

\subsection{Rappels}\label{rappeldet}
Soit $X$ variété projective complexe non singulière de dimension $n$ et $h_X$  une métrique $\cl$ hermitienne sur $TX$ et
$\omega$ sa
forme de Kähler associée. On considère pour tout $q=0,\ldots,n$, l'opérateur de Kodaira $\Delta_q$, agissant sur $A^{(0,q)}(X,E)$
et  on note par $\zeta_{\Delta_q}$ sa fonction Zêta. On sait que $\zeta_q$ s'étend méromorphiquement au plan complexe et holomorphiquement au voisinage de zéro. On définit alors le déterminant régularisé par:
\[
 \det(\Delta_q):=\exp(-\zeta_{\Delta_q}'(0)),
\]
Rappelons leur construction et l'énoncé de la conjecture: Pour tout $q\geq 0$, les espaces:
\[
 B^q=\overline{\pt}(A^{0,q-1}(X,E) )\subset A^{0,q}(X,E)\quad q\geq 1.
\]
$B^0=0$, et la fonction Zêta
\[
 \zeta_{B^q}(0)=\mathrm{Tr}(\Delta_q^{-s}|B^q)\quad \mathrm{Re}(s)>n.
\]
et on a,
\[
 \zeta_{\Delta_q}(s)=\zeta_{B^q}(s)+\zeta_{B^{q+1}}(s),
\]
et
\[
 \zeta_{B^{q+1}}(s)=\zeta_{\Delta_q}(s)-\zeta_{\Delta_{q-1}}(s)+\zeta_{\Delta_{q+2}}(s)+\ldots+(-1)^q\zeta_
{\Delta_0}(s).
\]
On définit:
\[
 D_q(E,h)=\exp(-\zeta_{B^q}'(0)).
\]

Dans \cite{Upper}, Gillet et Soulé énoncent la conjecture suivante:
\begin{conjecture}
 Il existe une constante $C_q(E)$ telle que, pour tout choix de métrique sur $E$, on a
\[
 D_q(E,h)\leq C_q(E)\quad  \forall \,  q\geq 1.
\]

\end{conjecture}
\begin{remarque}
\mbox{}
 \rm{

\begin{enumerate}
 \item $D_q(E,h)=D_q(E,th)$ pour tout réel $t>0$.
\item On a $D_q(\check{E},\check{h})=D_{n+1-q}(E,h)$, où $\check{E}=E^\ast\otimes K_X$ et $\check{h}$ est la métrique sur $\check{E}$ induite par $h$ et $h_X$.
\end{enumerate}

}

\end{remarque}

Rappelons que la torsion analytique holomorphe pour $(X,h_X)$ et $(E,h)$ est par définition:
\[
 T\bigl((X,\omega); (E,h) \bigr)=\sum_{q=0}^n(-1)^{q+1}q\zeta_{\Delta_q}'(0).
\]
On appelle métrique de Quillen sur le déterminant de cohomologie, $\la(E)$ la métrique suivante:
\[
h_{Q,h,\omega}=h_{L^2,h,\omega}e^{ T\bigl((X,\omega); (E,h) \bigr)},
\]
où $h_{L^2,h,\omega}$ est la métrique $L^2$ induite par $h$ et $\omega$.\\

On a,
{\allowdisplaybreaks
\begin{align*}
 T\bigl((X,h_X); (E,h)
\bigr)&=\sum_{q=0}^n(-1)^{q+1}q\bigl(\zeta_{B^q}'(0)+\zeta_{B^{q+1}}'(0)\bigr)\\
&=\sum_{q=0}^n(-1)^{q+1}q\zeta_{B^q}'(0)+\sum_{q=0}^n(-1)^{q+1}(q+1)\zeta_{B^{q+1}}'(0)-\sum_{q=0}^n(-1)^{q+1}\zeta_{B^{q+1}}'(0)\\
&=\sum_{q=0}^n(-1)^{q+1}q\zeta_{B^q}'(0)+\sum_{q=1}^n(-1)^{q+2}q
\zeta_{B^{q}}'(0)-\sum_{q=1}^n(-1)^{q}\zeta_{B^{q}}'(0)\quad (\text{on a}\; B^0=0,\;
B^{n+1}=0)\\
&=-\sum_{q=1}^n(-1)^q\zeta_{B^{q}}'(0).
\end{align*}}
Donc,
\[
 \exp\Bigl( -T\bigl((X,h_X); (E,h) \bigr) \Bigr)=\prod_{q=1}^nD_q(E,h)^{(-1)^{q+1}}.
\]
En particulier si $\dim_\C(X)=1$, alors
\begin{equation}\label{TmoinD}
 \exp\Bigl( -T\bigl((X,h_X); (E,h) \bigr) \Bigr)=D_1(E,h).
\end{equation}


\subsection{Sur la variation de la torsion analytique holomorphe}

\begin{theorem}
Soit $\p^n$ (resp. $X(\Si)$ une variété torique projective complexe non singulière de dimension $n$) muni d'une métrique de kähler $\omega$. Il existe $m_n\in \N$ tel que pour tout $m\in \N_{\geq 1}$,
\begin{align*}
 -T\bigl((\p^n,\omega);(\mathcal{O}(m),h^m) \bigr)
\leq-{c_m'}\quad \forall \, m\in \N_{\geq m_n}\quad \forall \, h\in \h_0,
\end{align*}
(resp.)
\begin{align*}
 -T\bigl((X(\Si),\omega);(E^m,h^m) \bigr)
\leq-{c_m'}\quad \forall \, m\in \N_{\geq m_n}\quad \forall \, h\in \h_0,
\end{align*}

où ${c_m'}$ est une constante réelle qui dépend uniquement de $m$ et de $\omega$.
\end{theorem}
\begin{proof}

Montrons  d'abord qu'on peut se ramener au cas $\omega=\omega_{FS}$. Soit $\omega$ une forme kählérienne $\cl$ quelconque sur  $T\p^n$ et $h_{\p^n}$ la métrique hermitienne associée.  Considérons un représentant pour de la
classe de Bott-Chern $\widetilde{Td}\bigl(T\p^n,h_{{}_{\p^n}},h_{{}_{\p^n,FS}} \bigr)$. Comme $\omega_{FS}$ est strictement positive et par compacité de $\p^n$, il existe $P$, un polynôme  en $\omega_{FS}$,  tel que
\[
0\leq  \widetilde{Td}\bigl(T\p^n,h_{{}_{\p^n}},h_{{}_{\p^n,FS}} \bigr)+P(\omega_{FS})\leq 2P(\omega_{FS}),
\]
dans $A(\p^n)$; l'algèbre des formes $(\ast,\ast)$-formes
différentielles sur $\p^n$. \\

On suppose en premier temps que $h$ est une métrique positive et de classe $\cl$ sur $\mathcal{O}(1)$. Comme
$c_1((\mathcal{O}(1),h))\geq 0$ par hypothèse, et que $\mathrm{ch}$ est une série à coefficients positifs, alors
\[
0\leq \mathrm{ch}\bigl(\mathcal{O}(m),h^m\bigr)\widetilde{Td}\bigl(T\p^n,h_{{}_{\p^n}},h_{{}_{\p^n,FS}}
\bigr)+ \mathrm{ch}\bigl(\mathcal{O}(m),h^m\bigr)P(\omega_{FS})\leq
2\mathrm{ch}\bigl(\mathcal{O}(m),h^m\bigr)P(\omega_{FS}).
\]

Par suite,
\[
-\int_{\p^n}\mathrm{ch}\bigl(\mathcal{O}(m),h^m\bigr)P(\omega_{FS})\leq\int_{\p^n}\mathrm{ch}\bigl(\mathcal{O}(m),h^m\bigr)
\widetilde{Td}\bigl(T\p^n,h_{{}_{\p^n}},h_{{}_{\p^n,FS}}
\bigr)\leq \int_{\p^n} \mathrm{ch}\bigl(\mathcal{O}(m),h^m\bigr)P(\omega_{FS}) .
\]
Or, d'après \cite{BGS1}, on dispose d'une formule donnant la variation de la métrique de Quillen lorsqu'on
change la métrique sur $T\p^n$ (resp. sur $TX(\Si)$):
\[
-\log\frac{h_{Q,h^m,\omega}}{h_{Q,h^m,\omega_{FS}}}=\int_{\p^n}ch\bigl(\mathcal{O}(m),h^m\bigr
)\widetilde{Td}(
T\p^n,h_{{}_{\p^n}},h_{{}_{\p^n,FS}}).
\]
 Par suite, pour toute métrique positive et de classe $\cl$ $h$ sur $\mathcal{O}(1)$, on a:
\[
-\int_{\p^n} \mathrm{ch}\bigl(\mathcal{O}(m),h^m\bigr)P(\omega_{FS})\leq-\log\frac{h_{Q,h^m,\omega_{FS}}}{h_{Q,h^m,\omega}}\leq \int_{\p^n}\mathrm{ch}\bigl(\mathcal{O}(m),h^m\bigr) P(\omega_{FS}).
\]
Comme  $\int_{\p^n}P(\omega_{FS})ch\bigl((\mathcal{O}(m),h^m)\bigr)$ ne dépend pas de $h$, alors on peut supposer dans la suite
que $\omega=\omega_{FS}$.\\

On munit $\p^n$ de la forme de Fubini-Study, $\omega_{FS}$. Soit $h$ une métrique positive de classe $\cl$ et
invariante par l'action du tore compact $(\s)^n$ sur $\mathcal{O}(1)$. Comme  $T\bigl((\p^n,\omega_{FS});(\mathcal{O}(m),(th)^m)
\bigr)=T\bigl((\p^n,\omega_{FS});(\mathcal{O}(m),h^m) \bigr)$, où $t>0$ fixé. Alors, on peut supposer que $h$ vérifie:
\[
 h\leq h_{FS}(\leq h_\infty),
\]
On en déduit que, $\omega_{FS}^n\leq \omega_\infty^n$ et donc:
\begin{equation}\label{hfsfty}
 h_{L^2,h^m,\omega_{FS}}\leq h_{L^2,h^m,\infty}.
\end{equation}

 On a
\[
-\log\frac{h_{Q,h^m,\omega_{FS}}}{h_{Q,h^m_\infty,\omega_{FS}}}=\int_{\p^n}\widetilde{\mathrm{ch}}\bigl(\mathcal{O}(m),h^m,
h_\infty^m \bigr)Td(\overline{T\p^n}_{FS}),
\]
où $h_{Q,h^m_\infty,\omega_{FS}}$ désigne la métrique de Quillen généralisée associée à $h_\infty^m$ et $\omega_{FS}$, voir
\cite{Mounir7}. Donc,
\begin{align*}
 -T\bigl((\p^n,\omega_{FS});(\mathcal{O}(m),h^m)
\bigr)&=\int_{\p^n}\widetilde{\mathrm{ch}}\bigl(\mathcal{O}(m),h^m,
h_\infty^m \bigr)Td(\overline{T\p^n}_{FS}) \\
&-T\bigl((\p^n,\omega_{FS});(\mathcal{O}(m),h^m_\infty) \bigr)+\log
\frac{h_{L^2,h^m,\omega_{FS}}}{h_{L^2,h^m_\infty,\omega_{FS}}}.
\end{align*}

 On montre que la classe de Bott-Chern associée à la suite métrisée d'Euler sur $\p^n$:
\[
0\lra \overline{\mathcal{O}}_0\lra \overline{\mathcal{O}(1)}_{{}_{FS}}^{\oplus n+1}\lra
\overline{T\p^n}_{{}_{FS}}\lra 0,
\]
est fermée, (voir par exemple \cite[proposition 5.3]{Character2}). On a donc, $Td(\overline{T\p^n}_{FS})=Td(\overline{\mathcal{O}(1)}_{FS})^{n+1}$.

On a
\begin{align*}
 \int_{\p^n}\widetilde{\mathrm{ch}}\bigl(\mathcal{O}(m),h^m,
h_\infty^m \bigr)Td(\overline{T\p^n}_{FS})&=\int_{\p^n}\widetilde{\mathrm{ch}}\bigl(\mathcal{O}(m),h^m,
h_\infty^m \bigr)Td(\overline{\mathcal{O}(1)}_{FS})^{n+1}\\
&=\int_{\p^n}\widetilde{\mathrm{ch}}\bigl(\mathcal{O}(m),h^m,
h_\infty^m \bigr)+ \int_{\p^n}\widetilde{\mathrm{ch}}\bigl(\mathcal{O}(m),h^m,
h_\infty^m \bigr)\Bigl(Td(\overline{\mathcal{O}(1)}_{FS})^{n+1}-1\Bigr)\\
&=2\int_{\p^n}\widetilde{\mathrm{ch}}\bigl(\mathcal{O}(m),h^m,
h_\infty^m \bigr)-\int_{\p^n}\widetilde{\mathrm{ch}}\bigl(\mathcal{O}(m),h^m,
h_\infty^m \bigr)\\
&+\int_{\p^n}\widetilde{\mathrm{ch}}\bigl(\mathcal{O}(m),h^m,
h_\infty^m \bigr)\Bigl(Td(\overline{\mathcal{O}(1)}_{FS})^{n+1}-1\Bigr)\\
&=2\int_{\p^n}\widetilde{\mathrm{ch}}\bigl(\mathcal{O}(m),h^m,
h_\infty^m \bigr)-\int_{\p^n}\widetilde{\mathrm{ch}}\bigl(\mathcal{O}(m),h^m,
h_\infty^m \bigr)\\
&+\sum_{j=1}^n b_j\int_{\p^n}\widetilde{\mathrm{ch}}\bigl(\mathcal{O}(m),h^m,
h_\infty^m \bigr)c_1(\overline{\mathcal{O}(1)}_{FS})^j.
\end{align*}
les $b_j$ sont les cofficients de la série formelle en $x$ suivante  $Td(x)=\frac{x}{1-e^{-x}}=1+b_1x+b_2x^2+\ldots$.\\

Rappelons que  $h_{L^2,h^m,\infty}$ est le volume $L^2$ définie dans \eqref{lafonctionelle}. Par ce qui précéde, on a:
{\allowdisplaybreaks
\begin{align*}
-T\bigl((\p^n,\omega_{FS});(&\mathcal{O}(m),h^m) \bigr)=\log h_{L^2,h^m,\infty}+2\int_{\p^n}\widetilde{\mathrm{ch}}\bigl(\mathcal{O}(m),h^m,
h_\infty^m \bigr)\\
&-\int_{\p^n}\widetilde{\mathrm{ch}}\bigl(\mathcal{O}(m),h^m,h_\infty^m \bigr)+\sum_{j=1}^n b_j\int_{\p^n}\widetilde{\mathrm{ch}}\bigl(\mathcal{O}(m),h^m,
h_\infty^m
\bigr)c_1(\overline{\mathcal{O}(1)}_{FS})^j\\
&+\log h_{L^2,h^m,\omega_{FS}}-\log h_{L^2,h^m,\infty}-\log h_{L^2,h^m_\infty,FS}-T\bigl((\p^n,\omega_{FS});(\mathcal{O}(m),h^m_\infty)
\bigr)\\
=&V_{\infty,m}(h)-\int_{\p^n}\widetilde{\mathrm{ch}}\bigl(\mathcal{O}(m),h^m,
h_\infty^m \bigr)+\sum_{j=1}^n b_j\int_{\p^n}\widetilde{\mathrm{ch}}\bigl(\mathcal{O}(m),h^m,
h_\infty^m \bigr)c_1(\overline{\mathcal{O}(1)}_{FS})^j\\
&+\log h_{L^2,h^m,\omega_{FS}}-\log h_{L^2,h^m,\infty}-\log \vc^2_{L^2,h^m_\infty,\omega_{FS}}-T\bigl((\p^n,\omega_{FS});(\mathcal{O}(m),h^m_\infty) 
\bigr)\\
\leq& \,c_m -\int_{\p^n}\widetilde{\mathrm{ch}}\bigl(\mathcal{O}(m),h^m,
h_\infty^m \bigr)+\sum_{j=1}^n b_j\int_{\p^n}\widetilde{\mathrm{ch}}\bigl(\mathcal{O}(m),h^m,
h_\infty^m \bigr)c_1(\overline{\mathcal{O}(1)}_{FS})^j\\
&-\log h_{L^2,h^m_\infty,\omega_{FS}}-T\bigl((\p^n,\omega_{FS});(\mathcal{O}(m),h^m_\infty)
\bigr)\quad \text{par}\quad \eqref{borneVmtheorem},\;\text{l'inégalité} \;\eqref{hfsfty}.
\end{align*}}
Donc, on a montré que pour tout $m\in \N_{\geq 1}$, et pour tout métrique $h$ positive et $\cl$ sur $E$:
\begin{equation}\label{inegaliteanomalie}
 -T\bigl((\p^n,\omega_{FS});(\mathcal{O}(m),h^m) \bigr)\leq {c_m'}
-\int_{\p^n}\widetilde{\mathrm{ch}}\bigl(\mathcal{O}(m),h^m,
h_\infty^m \bigr)+\sum_{j=1}^n b_j\int_{\p^n}\widetilde{\mathrm{ch}}\bigl(\mathcal{O}(m),h^m,
h_\infty^m \bigr)c_1(\overline{\mathcal{O}(1)}_{FS})^j,
\end{equation}
où on a posé ${c_m'}:=c_m-\log
h_{L^2,h^m_\infty,\omega_{FS}}-T\bigl((\p^n,\omega_{FS});(\mathcal{O}(m),h^m_\infty)
\bigr)$. D'après \cite{Mounir7}, l'inégalité précédente s'étend à $\h_0$.\\

On se propose maintenant d'établir que la fonctionnelle suivante est majorée sur $\h_0$ pour $m$ assez grand:
\[-\int_{\p^n}\widetilde{\mathrm{ch}}\bigl(\mathcal{O}(m),h^m,
h_\infty^m \bigr)+\sum_{j=1}^n b_j\int_{\p^n}\widetilde{\mathrm{ch}}\bigl(\mathcal{O}(m),h^m,
h_\infty^m \bigr)c_1(\overline{\mathcal{O}(1)}_{FS})^j,
\]

Pour cela on aura besoin  de montrer que:
\[
 \int_{\p^n}\widetilde{\mathrm{ch}}\bigl(\mathcal{O}(1),h,
h_\infty \bigr)c_1(\overline{\mathcal{O}(1)}_{FS})^j\leq j!\, 2^{n+1}\int_{\p^n} \widetilde{\mathrm{ch}}\bigl(\mathcal{O}(1),h,
h_\infty \bigr)\quad \forall \, j=1,\ldots,n.
\]
En fait, cela résulte d'un fait plus général prouvé dans le lemme suivant:
\begin{lemma}
 Soit $h\leq h_0\leq h_\infty$ trois métriques hermitiennes, avec $h,h_0\in \h_0$.
On a,
\[
 \int_{\p^n}\widetilde{\mathrm{ch}}\bigl(\mathcal{O}(1),h,
h_\infty \bigr)c_1(\overline{\mathcal{O}(1)}_0)^j\leq j!\, 2^{n+1}\int_{\p^n} \widetilde{\mathrm{ch}}\bigl(\mathcal{O}(1),h,
h_\infty \bigr)\quad \forall \, j=1,\ldots,n.
\]
\end{lemma}
\begin{proof}
Montrons ce lemme. On a
{\allowdisplaybreaks
\begin{align*}
 \int_{\p^n}\widetilde{\mathrm{ch}}\bigl(&\mathcal{O}(1),h,h_\infty)\mathrm{ch}\bigl(\mathcal{O}(m),h_0^m\bigr)=\int_{\p^n}
\widetilde{\mathrm{ch}}\bigl(\mathcal{O}(m+1),h\otimes h_0^m,h_\infty\otimes h_0^m)\\
&=\int_{\p^n}\widetilde{\mathrm{ch}}\bigl(\mathcal{O}(m+1),h\otimes h_0^m,h_\infty^{m+1}\bigr)-\int_{\p^n}
\widetilde{\mathrm{ch}}\bigl(\mathcal{O}(m+1),h_\infty\otimes h_0^m,h_\infty^{m+1}\bigr)\\
&=(1+m)^{n+1}\int_{\p^n}\widetilde{\mathrm{ch}}\bigl(\mathcal{O}(1),(h\otimes h_0^m)^{\frac{1}{m+1}},h_\infty\bigr)-(1+m)^{n+1}\int_{\p^n}
\widetilde{\mathrm{ch}}\bigl(\mathcal{O}(1),(h_\infty\otimes h_0^m)^{\frac{1}{m+1}},h_\infty\bigr)\\
&=(1+m)^{n+1}\int_{\Delta_n}\Bigl(\frac{f+mf_0}{1+m}\Bigr){}^{\check{}}(x)dx-(1+m)^{n+1}\int_{\Delta_n}\Bigl(\frac{f_\infty+mf_0}{1+m} \Bigr){}^{\check{}}(x)dx\quad \text{d'après}\; \eqref{bottchernfenchel}.
\end{align*}}
Comme on a supposé que $f_h\leq f_0:=\log \vc_0(\exp(-\cdot))\leq f_\infty$; alors
$\frac{f_h+mf_0}{1+m}\geq f_h$, et $\frac{f_\infty+mf_0}{1+m}\leq f_\infty$,  donc
\begin{equation}\label{cherninffff}
\int_{\p^n}\widetilde{\mathrm{ch}}\bigl(\mathcal{O}(1),h,h_\infty)\mathrm{ch}\bigl(\mathcal{O}(m),h_0^m\bigr)\leq
(1+m)^{n+1}\int_{\Delta_n}\check{f}_h(x)dx\quad \forall \,m\in \N.
\end{equation}

On a $\int_{\p^n}\widetilde{\mathrm{ch}}\bigl(\mathcal{O}(1),h,h_\infty)\mathrm{ch}\bigl(\mathcal{O}(m),h_0^m\bigr)$ est un
polynôme à coefficients positifs. En effet, la forme généralisée $\widetilde{\mathrm{ch}}\bigl(\mathcal{O}(1),h,h_\infty)$ est
positive, puisqu'elle est localement
somme de termes de la forme \[\bigl(-\log\frac{h}{h_\infty}\bigr)c_1(\mathcal{O}(1),h)^j c_1(\mathcal{O}(1),h_\infty)^{n-j}\quad
j=0,\ldots,n,\]
les métriques sont positives,  $h\leq h_\infty$ et $\mathrm{ch}\bigl(\mathcal{O}(m),h_0^m\bigr)$ est positif aussi. On
 écrit:
\begin{align*}
\int_{\p^n}\widetilde{\mathrm{ch}}\bigl(\mathcal{O}(1),h,h_\infty)\mathrm{ch}\bigl(\mathcal{O}(m),h_0^m\bigr)&=\sum_{k=0}^n
\frac{m^k}{k!}\int_{\p^n}\widetilde{\mathrm{ch}}\bigl(\mathcal{O}(1),h,h_\infty)c_1(\mathcal{O}(1),h_0)\\
&=\sum_{k=0}^n \frac{q_k}{k!}m^k,
\end{align*}
avec $q_k:=\int_{\p^n}\widetilde{\mathrm{ch}}\bigl(\mathcal{O}(1),h,h_\infty)c_1(\mathcal{O}(1),h_0)^k$ pour
$k=0,\ldots,n$, et on pose $q:=\int_{\Delta_n}\check{f}(x)dx$.
 Sous ces notations, \eqref{cherninffff} s'écrit:
\[
 \sum_{k=0}^n \frac{q_k}{k!}m^k\leq (1+m)^{n+1}q\quad \forall \,m\in \N.
\]

On déduit que $\frac{q_k}{k!}m^k\leq \sum_{j=0}^n \frac{q_j}{j!}m^j\leq q(1+m)^{n+1}$, pour tout $m\in
\N$. En particulier  pour $m=1$, on obtient $q_k\leq k! 2^{n+1}q $
pour tout $k=1,\ldots,n$. Notons que pour $j=0$, si on prend $m=0$ alors $q_0\leq q$.

Récapitulons, on a montré que
\begin{equation}\label{bornesurcoefficients}
 q_0\leq q,\quad q_j\leq j!2^{n+1}q\quad \forall\, j=1,\ldots,n.
\end{equation}
\end{proof}
D'après \eqref{inegaliteanomalie}, et en prenant $h_0=h_{FS}$, et par \eqref{bornesurcoefficients}, on
obtient:
\begin{align*}
 -T\bigl((\p^n,\omega_{FS});(\mathcal{O}(m),h^m) \bigr)&\leq {c_m'}
-m^{n+1}\int_{\p^n}\widetilde{\mathrm{ch}}\bigl(\mathcal{O}(1),h,
h_\infty \bigr)+\sum_{j=1}^n b_j m^{n+1-j}\int_{\p^n}\widetilde{\mathrm{ch}}\bigl(\mathcal{O}(1),h,
h_\infty \bigr)c_1(\overline{\mathcal{O}(1)}_{FS})^j\\
&={c_m'} -m^{n+1}q+\sum_{j=1}^n \frac{q_j}{j!} b_j m^{n+1-j}\quad \text{notons que}\,\, q=\int_{\p^n}\widetilde{\mathrm{ch}}\bigl(\mathcal{O}(1),h,
h_\infty \bigr) \\
&\leq {c_m'} -m^{n+1}q+\sum_{j=1}^n 2^{n+1}q |b_j| m^{n+1-j}\\
&=\Bigl(-m^{n+1}+\sum_{j=1}^n 2^{n+1} |b_{n+1-j}| m^{j}\Bigr)q-{c_m'},
\end{align*}
pour tout $m\in \N$ et $ h\leq h_{FS}$, positive et invariante par l'action du tore compact $(\s)^{n}$.\\

On peut trouver $m_0\in \N$ tel que $-m^{n+1}+\sum_{j=1}^n 2^{n+1} |b_{n+1-j}| m^{j}\leq 0$ pour tout $m\geq m_0$.  Comme $q\geq 0$, on conclut que
\[
 -T\bigl((\p^n,\omega_{FS});(\mathcal{O}(m),h^m) \bigr)\leq -{c_m'}\quad \forall\, m\geq
m_0.
\]

\end{proof}

\begin{remarque}
En dimension 1, on obtient:
\[
D(\mathcal{O}(m),h^m)\leq -{c_m'}\quad \forall\, h\in \h_0\quad \forall \, m\geq 1.
\]
\end{remarque}
\begin{Corollaire}
Sous ces hypothèses, on a
\[
\prod_{q=0}^{[\frac{n}{2}]}D_{2q}(\mathcal{O}(m),h^m)\leq -{c_m'} \prod_{q=0}^{[\frac{n-1}{2}]}D_{2q+1}(\mathcal{O}(m),h^m),
\]
pour tout $h\in \h_0$.
\end{Corollaire}

\section{Rappel sur la fonctionnelle $\mathcal{F}_{\omega_0}$  et résultats antérieurs connexes}\label{RRR}

Dans ce paragraphe, on introduit deux fonctionnelles classiques sur l'espace de métriques positives sur un fibré en droites ample
sur une variété kählérienne compacte. Soit $X$ une variété compacte kählérienne. Si
$L$ est un fibré en droites ample sur $X$, alors il existe une forme de Kähler $\omega_0$ dans la première
classe de Chern $c_1(L)$. On pose
\[
 \mathcal{H}_{\omega_0}=\bigl\{ u\in \mathcal{C}^\infty(X)\, \bigl|\, \omega_u:=dd^c u+\omega_0>0 \bigr\},
\]
cet ensemble s'identifie à l'ensemble des métriques  $\cl$ définies positives sur $L$. Classiquement, on
définit deux fonctionnelles, $\mathcal{E}_{\omega_0}$ et $\mathcal{L}_{\omega_0}$, sur
$\mathcal{H}_{\omega_0}$. La première fonctionnelle  $\mathcal{E}_{\omega_0}$, appelée la fonctionnelle
d'énergie  qui apparaît dans les travaux d'Aubin cf. \cite{Aubin2} et Mabuchi cf. \cite{Mabuchi}, elle
se définit comme suit:
\[
\mathcal{E}_{\omega_0}(u):=\frac{1}{(n+1)!\mathrm{Vol}(\omega_0)}\sum_{j=0}^n\int_X u\bigl(dd^c u+\omega_0
\bigr)^j\wedge
\omega_0^{n-j}.
\]
On a, $\omega_0$ définit une métrique kählérienne sur $X$, et
donc sur $K_X$. Soit $u\in \mathcal{H}_{\omega_0}$, $u$ définit donc une métrique sur $L$. Cela nous donne une métrique sur
$H^0(X,L\otimes
K_X)$. La fonctionnelle $\mathcal{L}_{\omega_0}$ est alors définie comme suit:
\[
 \mathcal{L}_{\omega_0}(u):=-\frac{1}{N}\log \det\bigl(\bigl<s_i,s_j\bigr> \bigr)_{1\leq i,j\leq N },
\]
avec $s_1,\ldots,s_N$ est un ensemble orthogonal de sections globales de $L\otimes K_X$  pour la métrique définie par
$\omega_0$ et  formant une base pour $H^0(X,L\otimes K_X)$, voir \cite{Berman}.\\

On pose,
\[
\mathcal{F}_{\omega_0}=\mathcal{E}_{\omega_0}-\mathcal{L}_{\omega_0}.
\]

$L$ est supposé ample, c'est à dire que $L$ admet une métrique $h_0$ de classe $\cl$ telle que $\omega_0:=c_1\bigl( L,h_0\bigr)$ soit une forme de Kähler. On pose
\[
 \mathcal{H}_{\omega_0}:=\bigl\{u\in \mathcal{C}^\infty(X)\,\bigl|\; \omega_u:=dd^cu +\omega_0>0 \bigr\}
\]
cet ensemble s'identifie avec l'espace des métriques hermitiennes de classe $\cl$ et positives sur $L$.\\

Sur $ \mathcal{H}_{\omega_0}$, on considère deux fonctionnelles

\begin{enumerate}

\item La fonctionnelle d'énergie, $\mathcal{E}_{\omega_0}$:
\[
\mathcal{E}_{\omega_0}(u):=\frac{1}{(n+1)!V}\sum_{i=0}^n\int_X u\bigl(dd^c+\omega_0 \bigr)^j\wedge
\omega_0^{n-j}=\frac{1}{V}\int_X\widetilde{\mathrm{ch}}\bigl(L,h_u,h_0 \bigr),
\]
où $V=\int_X\frac{\omega_0^n}{n!}$.\\
\item La fonctionnelle $\mathcal{L}_{\omega_0}$, voir \cite[§ 1.2]{Berman}:

On munit le fibré en droites canonique $K_X$,  de la métrique induite par $\omega_0$. On pose:
\[
 \mathcal{L}_{\omega_0}(u):=-\frac{1}{N}\log \det\bigl(\bigl<s_i,s_j\bigr> \bigr)_{1\leq i,j\leq N },
\]
$N=\dim H^0(X,L+K_X)$, supposé non nul, où $\bigl\{s_i \bigr\}$ est une base de $H^0(X,L+K_X)$, orthogonale pour la métrique
induite par $\omega_0$.\\
\end{enumerate}

Dans \cite{Berman}, Berman considère  la fonctionnelle suivante sur $\mathcal{H}_{\omega_0}$:
\begin{equation}\label{bermanfonctionnelle}
 \mathcal{F}_{\omega_0}:=\mathcal{E}_{\omega_0}-\mathcal{L}_{\omega_0},
\end{equation}
qui est invariante par addition de constantes.\\

Supposons que $(X,L)$ est $K$-homogène, c'est à dire il existe un groupe de Lie compact semi-simple $K$ qui agit transitivement sur $X$, et que cette action se relève à $L$. Soit $\omega_0$ l'unique forme de Kähler invariante par l'action de $K$ sur $X$.

\begin{theorem}\label{BermanTheorem}
Soit $L$ un fibré en droites holomorphe $K$-homogène sur $X$, alors

\[
 \mathcal{F}_{\omega_0}(u)\leq 0\quad \forall \, u\in \mathcal{H}_{\omega_0},
\]

 avec égalité lorsque $u$ est constante, modulo l'action de $Aut_0(X,L)$.
\end{theorem}
\begin{proof}
 Voir \cite[corollaire 2]{Berman}.
\end{proof}

\subsection{Comparaison de $\mathcal{F}_{\omega_0}$ et de $V_{\infty,\ast}$}\label{CompBerman}
On se propose dans ce paragraphe de comparer $V_{\infty,m}$, la fonctionnelle introduite dans ce texte, avec
l'approche de \cite{Berman}.\\

On considère $\p^n$ muni d'une forme de Kähler  $\omega_0:=c_1(O(1),h_0)$ invariante par $(\s)^n$. D'après \cite{Berman}, on a
\begin{align*}
 \mathcal{F}_{m\omega_0}(u)=\frac{1}{\binom{n+m-2}{n}}\log \prod_{\nu
\in\mathcal{P}_{m-2}}<z^\nu,z^\nu>_{L^2_{h^m,\omega_0}}+\frac{1}{V(\mathcal{O}(m))}\int_{\p^n}\widetilde{ch
}\bigl(\mathcal{O}(m),h^m,h^m
_0\bigr)\leq 0, \quad \forall\, m\in \N_{\geq 1}
\end{align*}
pour toute $h\in \mathcal{H}_{\omega_0}$, où on a posé $u=-m\log\frac{h}{h_0}$. \\

\begin{theorem}
Il existe $N\in \N$ et  une constante $c$ qui dépend uniquement de $h_0$ telle que
\[
 \mathcal{F}_{m\omega_0}(u_h)\leq \frac{1}{\binom{n+m-2}{n}}V_{m-2}(h)+c,\quad \forall \, h\in
\h_0\quad \forall\, m\geq N.
\]
où $u_h:=-m\log\frac{h}{h_{FS}}$.
\end{theorem}

\begin{proof}
Notons qu'il existe donc $N\in \N$ tel que pour tout $m\geq N$, on a
$\frac{n!\binom{n+m-2}{n}m^{n+1}}{m^n}\leq 2(m-2)^{n+1}$. En effet, par la formule de Stirling, on trouve que $
 \lim_{m\mapsto \infty}\frac{n!\binom{n+m-2}{n}}{m^n}=\frac{1}{e^n}$.\\

Comme $\p^n$ est compact, alors il existe $t_0>0$ tel que $\omega_{FS}^n\leq t_0\omega_\infty^n$. Soit $m\geq N$, on a
\begin{align*}
 \frac{1}{\binom{n+m-2}{n}}\log
&\prod_{\nu\in\mathcal{P}_{m-2}}<z^\nu,z^\nu>_{L^2_{h^m,\omega_0}}+\frac{1}{V(\mathcal{O}(m))}\int_{\p^n}\widetilde{\mathrm{ch}}
\bigl(\mathcal{O}(m),h^m,h^m_0\bigr)\\
&\leq  \frac{1}{\binom{n+m-2}{n}}\log
\prod_{\nu\in \mathcal{P}_{m-2}}<z^\nu,z^\nu>_{L^2_{h^m,\infty}}+\frac{1}{V(\mathcal{O}(m))}\int_{\p^n}\widetilde{\mathrm{ch}}
\bigl(\mathcal{O}(m),h^m,h^m_\infty\bigr)\\
&-\frac{1}{V(\mathcal{O}(m))}\int_{\p^n}\widetilde{\mathrm{ch}}
\bigl(\mathcal{O}(m),h^m_0,h^m_\infty\bigr)+\log t_0\\
&\leq \frac{1}{\binom{n+m-2}{n}}\biggl(  \log \prod_{\nu\in
\mathcal{P}_{m-2}}<z^\nu,z^\nu>_{L^2_{h^m,\infty}}+\frac{\binom{n+m-2}{n}m^{n+1}}{\frac{m^n}{n!}}\int_{\p^n
}\widetilde{\mathrm{ch}}\bigl
(\mathcal{O}(1),h,h_\infty \bigr)\biggr)\\
&-\frac{1}{V(\mathcal{O}(m))}\int_{\p^n}\widetilde{\mathrm{ch}}
\bigl(\mathcal{O}(m),h^m_0,h^m_\infty\bigr)+\log t_0\\
&\leq \frac{1}{\binom{n+m-2}{n}}\biggl(  \log \prod_{\nu\in \mathcal{P}_{m-2}}<z^\nu,z^\nu>_{L^2_{h^m,\infty}}+2(m-2)^{n+1}\int_{\p^n}\widetilde{\mathrm{ch}}\bigl
(\mathcal{O}(1),h,h_\infty \bigr)\biggr)\\
&-\frac{1}{V(\mathcal{O}(m))}\int_{\p^n}\widetilde{\mathrm{ch}}
\bigl(\mathcal{O}(m),h^m_0,h^m_\infty\bigr)+\log t_0\\
&=\frac{1}{\binom{n+m-2}{n}}V_{m-2}(h)-\frac{1}{V(\mathcal{O}(m))}\int_{\p^n}\widetilde{\mathrm{ch}}
\bigl(\mathcal{O}(m),h^m_0,h^m_\infty\bigr)+\log t_0.\\
\end{align*}
Comme $\frac{1}{V(\mathcal{O}(m))}\int_{\p^n}\widetilde{\mathrm{ch}}
\bigl(\mathcal{O}(m),h^m_0,h^m_\infty\bigr)\geq 0$, par conséquent:

\[
 \frac{1}{\binom{n+m-2}{n}}\log
\prod_{\nu\in\mathcal{P}_{m-2}}<z^\nu,z^\nu>_{L^2_{h^m,\omega_0}}+\frac{1}{V(\mathcal{O}(m))}\int_{\p^n}\widetilde{\mathrm{ch}}
\bigl(\mathcal{O}(m),h^m,h^m_0\bigr)\leq \frac{1}{\binom{n+m-2}{n}}V_{m-2}(h)+\log t_0.
\]
Donc, en posant $u_h=-\log \frac{h}{h_{FS}}$, on a montré que
\[
\mathcal{F}_{m\omega_0}(u_h)\leq  \frac{1}{\binom{n+m-2}{n}}V_{m-2}(h)+\log t_0\quad \forall\, h\in
\h_0\quad \forall\, m\geq N.
\]

\end{proof}
\begin{Corollaire}
Si $V_{\infty,m}$ est majorée alors $\mathcal{F}_{m\omega_0}$ l'est aussi.
\end{Corollaire}
\begin{proof}
Cela résulte directement du théorème précédent.
\end{proof}

Pour tout $m\in \N_{\geq 1}$, on note par $\mathcal{H}_{m\omega_0,0}$ le sous ensemble de $\mathcal{H}_{m\omega_0}$ qui correspond aux
métriques positives $\cl$ invariantes par $(\s)^n$ sur $\mathcal{O}(m)$. On retrouve  une version faible  du \eqref{BermanTheorem}:
\begin{Corollaire}
 Il existe $N\in \N$ tel que
\[
 \mathcal{F}_{m\omega_0}(u)\leq c''_m\quad \forall \, u\in \mathcal{H}_{m\omega_0,0} \quad\forall \, m\geq
N,
\]
où $c''_m$ est une constante réelle qui dépend uniquement de $m$ et $\omega_0$.
\end{Corollaire}

\begin{proof}

On a montré que, dans la proposition précédente que:
\[
\mathcal{F}_{m\omega_0}(u_h)\leq  \frac{1}{\binom{n+m-2}{n}}V_{m-2}(h)+\log t_0\quad \forall\, h\in
\h_0\quad \forall\, m\geq N.
\]
D'après \eqref{borneVmtheorem}, on  déduit que
\[\mathcal{F}_{m\omega_0}(u_h)\leq \frac{1}{\binom{n+m-2}{n}}c_{m-2}+\log t_0.
\]

\end{proof}

\bibliographystyle{plain}
\bibliography{biblio}

\vspace{1cm}

\begin{center}
{\sffamily \noindent National Center for Theoretical Sciences, (Taipei Office)\\
 National Taiwan University, Taipei 106, Taiwan}\\

 {e-mail}: {hajli@math.jussieu.fr}

\end{center}

\end{document}